\documentclass{amsart}
\usepackage{appendix}
\usepackage{tabularx}
\usepackage[numbers,sort&compress]{natbib}
\usepackage{color}
\usepackage[colorlinks=true,linkcolor=blue,citecolor=blue,urlcolor=blue]{hyperref}
\newtheorem{theorem}{Theorem}[section]
\newtheorem{lemma}[theorem]{Lemma}
\newtheorem{corollary}[theorem]{Corollary}

\theoremstyle{definition}

\theoremstyle{remark}
\newtheorem{remark}[theorem]{Remark}

\newcommand{\Rm}{\operatorname{Rm}}
\newcommand{\Ric}{\operatorname{Ric}}
\newcommand{\Vol}{\operatorname{Vol}}

\newcommand{\dd}{\,\mathrm{d}}

\newcommand{\cK}{\mathcal{K}}
\newcommand{\cM}{\mathcal{M}}

\numberwithin{equation}{section}
\allowdisplaybreaks[4]

\begin{document}

\title[Long-time existence of the Laplacian flow ]{Long-time existence of the Laplacian flow with Weyl tensor}


\author{Chuanhuan Li}
\address{School of Mathematical Sciences, Laboratory of Mathematics and Complex Systems, Beijing Normal University, Beijing 100875, China \newline
${\quad}$Shanghai Institute for Mathematics and Interdisciplinary Sciences, Shanghai 200433, China}
\curraddr{}
\email{chli@simis.cn}
\thanks{}

\author{Yi Li}
\address{Shanghai Institute for Mathematics and Interdisciplinary Sciences, Shanghai 200433, China \newline
${\quad}$ Fudan University, Shanghai 200433, China}
\curraddr{}
\email{yilicms@simis.cn, yilicms@gmail.com}
\thanks{}

\author{Kai Zheng}
\address{University of Chinese Academy of Sciences, Beijing 100190, China}
\curraddr{}
\email{KaiZheng@amss.ac.cn}
\thanks{}



\date{}

\dedicatory{}

\begin{abstract}
We establish integral extension criteria for the Laplacian flow
of closed $G_2$-structures on compact $7$-dimensional manifolds $M$.
Assuming a uniform lower bound for the scalar curvature, we prove
that a finite-time solution extends smoothly whenever the spatial
$L^p$ norm of the $77$-dimensional Weyl component $W_{77}$ remains
uniformly bounded for some $p\ge 7/2$.
We also obtain a mixed space--time extension criterion under the
condition
\[
W_{77}\in L^q(0,T;L^p(M,g(t))),
\qquad
\frac{7}{2}<p<\infty,\quad
1\le q<\infty,\quad
\frac{1}{q}+\frac{7}{p}\le 2.
\]
In particular, finite space-time $L^4$ integrability of $W_{77}$
is sufficient for extension.

\end{abstract}

\maketitle

\section{Introduction}

In order to study a smooth $7$-manifold $M$ admitting a {torsion-free} $G_{2}$-structure, Bryant \cite{Bry06} introduced the following Laplacian $G_2$ flow for closed $ G_{2}$-structure:
\begin{equation}
  \left \{
       \begin{array}{rl}
          \partial_{ t}\varphi(t)&=\Delta_{\varphi(t)}\varphi(t),\\
           d\varphi(t)&=0,\\
           \varphi(0)&=\varphi,
       \end{array}
  \right.
  \label{eq:flow}
\end{equation}
where $\Delta_{\varphi(t)}=dd^{\ast}_{\varphi(t)}+d^{\ast}_{\varphi(t)}d$ is the Hodge Laplacian of $g(t)$ and $\varphi$ is an initial  closed $G_{2}$-structure({a positive closed three-form).} Here $g(t)$ is the Riemannian metric {algebraically determined by} $\varphi(t)$. Since $\Delta_{\varphi}\varphi=dd^{\ast}_{\varphi}\varphi$ for a closed $G_{2}$-structure $\varphi$, we see that the closedness of $\varphi(t)$ is preserved along the Laplacian flow $\eqref{eq:flow}$. 

The flow $\eqref{eq:flow}$ can be viewed as the gradient flow for the Hitchin functional when the variations are restricted to the cohomology class of the closed $G_{2}$-structure, where Hitchin functional was introduced by Hitchin \cite{Hit00}
$$\mathcal{H}:[\overline{\varphi}]_{+}\longrightarrow\mathbb{R}^{+},\ \varphi\longmapsto\frac{1}{7}\int_{M}\varphi\wedge\psi=\int_{M}\ast_{\varphi}1.$$
Here $\overline{\varphi}$ is a closed $G_{2}$-structure on $M$ and  $[\overline{\varphi}]_{+}$ is the open subset of the cohomology class $[\overline{\varphi}]$ consisting of $G_{2}$-structures. Any critical point of $\mathcal{H}$ gives rise to a torsion-free $G_{2}$-structure {when $7$-manifold $M$ is compact.} 

The stationary points of the Laplacian flow $\eqref{eq:flow}$ are harmonic $\varphi$, which on a compact manifold are the torsion-free $G_{2}$-structure. Hence the long-time behaviour of the flow is an important topic. In \cite{BX11}, applying De Turck's trick and Hamilton's Nash-Moser inverse function theorem, Bryant and Xu proved the short-time existence for $\eqref{eq:flow}$. In 2017, Lotay and Wei \cite{LW17} proved the long-time existence for the Laplacian $G_{2}$ flow $\eqref{eq:flow}$ under bounded Riemann curvature or Ricci curvature. In 2021, the second author \cite{local curvature of G2} gave {an alternative proof via local curvature estimates, showing that the flow persists as long as the Ricci curvature remains bounded.}

${}$

Lotay and Wei \cite{LW17} propose the following question (see also \cite{LL26A,Lot20}):
\begin{quote}
    {\bf Question:} {\it whether the Laplacian $G_{2}$ flow on closed $G_{2}$-structures will exist as long as the torsion tensor or scalar curvature remains uniformly bounded.}
\end{quote}
Fine and Yao \cite{FY18} obtained the long-time existence of the Laplacian flow $\eqref{eq:flow}$ with bounded scalar curvature when $M=X\times \mathbb{T}^{3}$, where $X$ is a $4$-manifold with hypersymplectic structure and $\mathbb{T}^{3}$ is the standard flat 3-torus. Let $B \subset \mathbb{R}^3$ be a bounded domain with smooth boundary, Lambert and Lotay \cite{LL21} prove the long-time existence on $B\times \mathbb{T}^{4}$ without bounded scalar curvature assumptions. Picard and Suan \cite{PS24} also proved that the Laplacian $G_{2}$ flow $\eqref{eq:flow}$ exists for all time under the suitable initial $\varphi_{0}$ when $M=\mathbb{T}^{3}\times CY^{2}$ or $M=\mathbb{S}^{1}\times CY^{3}$ with $CY^{2},CY^{3}$ are compact Calabi–Yau $2,3$-folds.

In \cite{LL26}, the authors derive a pinching estimate for the traceless Ricci curvature in terms of scalar curvature and the $C^{1}$  norm of the Weyl tensor under the Laplacian flow for closed $G_{2}$ structures. Using this, they prove the long-time existence under uniformly bounded scalar curvature $R(g(t))$ and the $C^{1}$  norm of Weyl tensor  $W(g(t))$.

In fact, Cleyton and Ivanov \cite{CI08} showed that the  Weyl tensor can be decomposed into the following three parts by $G_{2}$ structure:(see also \cite{DGK25} for details)
$$W=W_{77}+W_{64}+W_{27}.$$
Thus, it is worth investigating which part, together with the scalar curvature, determines the extension of the Laplacian flow.

In this paper, we study the extension of the Laplacian flow with bounded integral of Weyl tensor, with particular emphasis on the component $W_{77}$. We give the following result:
\begin{theorem}\label{thm1.1}
Let $(M^7,\varphi(t))_{t\in[0,T)}$ be a solution of
\eqref{eq:flow} on a compact manifold, where $0<T<\infty$. For some $p\geq 7/2$,  assume
that the scalar curvature and the $L^{p}$ norm of $W_{77}$ satisfy
$$
  \sup_{M\times[0,T)}(-R)\leq K<\infty,
$$
and
$$
  \sup_{0\leq t<T}\int_M |W_{77}(t)|^p\,\dd V_{g(t)}\leq A<\infty.
$$
Then $\varphi(t)$ extends smoothly past time $T$.
\end{theorem}
Besides, we also prove the long-time existence under bounded scalar curvature and mixed spacetime norm of $W_{77}$.
\begin{theorem}\label{thm1.2}
Let $(M^7,\varphi(t))_{t\in[0,T)}$ be a solution of
\eqref{eq:flow} on a compact manifold, where $0<T<\infty$. 
Let $\frac{7}{2}< p<\infty$ and $1\leq q<\infty$ satisfy
$$\frac{1}{q}+\frac{7}{p}\leq 2.$$
Suppose that the scalar curvature and $W_{77}$ satisfy
$$
  \sup_{M\times[0,T)}(-R)\leq K<\infty,
$$
and  
$$
  \int_{0}^{T}\left(\int_{M}|W_{77}|^{p}\dd V_{g(t)}\right)^{\frac{q}{p}}\dd t \leq A<\infty.
$$
Then $\varphi(t)$ extends smoothly past time $T$.
\end{theorem}

\begin{remark}
    If we let $p=q=4$, then the mixed spacetime norm of $W_{77}$ is
    $$\int_{0}^{T}\int_{M}|W_{77}|^{4}\dd V_{g(t)}\dd t \leq A<\infty.$$
\end{remark}
  
For the diagonal case \(p=q=4\), we can weaken the assumption of $W_{77}$.
\begin{corollary}\label{1.4}
    Let $(M^7,\varphi(t))_{t\in[0,T)}$ be a solution of
\eqref{eq:flow} on a compact manifold, where $0<T<\infty$. 
Let $f:[0,\infty)\rightarrow[1,\infty)$ be finite-valued and nondecreasing, with
\begin{align}
    \int_{1}^{\infty}\frac{\dd s}{sf(s)}=\infty.\notag
\end{align}
Suppose that the scalar curvature and $W_{77}$ satisfy
$$
  \sup_{M\times[0,T)}(-R)\leq K<\infty
$$
with $K>0$ and  
\begin{align}
    \int_{0}^{T}\int_{M}\frac{|W_{77}|^{4}}{f(|W_{77}|/K)}\dd V_{g(t)}\dd t \leq A<\infty.\notag
\end{align}
Then $\varphi(t)$ extends smoothly past time $T$.
\end{corollary}
\begin{remark}
    We can let $f(x)=\log(e+x)$, then the assumption in Corollary \ref{1.4} is 
    \begin{align}
    \int_{0}^{T}\int_{M}\frac{|W_{77}|^{4}}{\log(e+|W_{77}|/K)}\dd V_{g(t)}\dd t \leq A<\infty.\notag
\end{align}
\end{remark}

We give an outline of this paper. We review the basic theory in Section \ref{section2} about $G_{2}$-structure and curvature tensor of closed $G_{2}$-structures. In Section \ref{section3}, we prove the spatial $L^p$ extension criterion, which can be divided into two parts: $p>\frac{7}{2}$ in Section \ref{section3.1} and $p=\frac{7}{2}$ in Section \ref{section3.2}. In Section \ref{section4}, we also prove the long-time existence of the Laplacian flow under bounded scalar curvature and mixed spacetime $L^{q}_{t}L^{p}_{x}$ of $W_{77}$.

\section{$G_{2}$-structure and curvature tensors}\label{section2}
In this section, we give some basis theory of $G_{2}$-structures in \cite{Bry06,Spiros introduce to G2,LW17}.

\subsection{$G_{2}$-structure on smooth $7$-manifold}\label{subsection2.1}

Let $\{e_{1},e_{2},\cdots,e_{7}\}$ denote the standard basis of $\mathbb{R}^{7}$ and let $\{e^{1},e^{2},\cdots,e^{7}\}$ be its dual basis. Define the 3-form
$$\phi:=e^{123}+e^{145}+e^{167}+e^{246}-e^{257}-e^{347}-e^{356},$$
where $e^{ijk}:=e^{i}\wedge e^{j}\wedge e^{k}$. The subgroup of ${\rm GL}(7,\mathbb{R})$ {that fixes $\phi$} is the exceptional Lie group $G_{2}$, a compact, connected, simple $14$-dimensional Lie subgroup of ${\rm  SO}(7)$. In fact, $G_{2}$ acts irreducibly on $\mathbb{R}^{7}$ and preserves the metric and orientation for which $\{e_{1},e_{2},\cdots,e_{7}\}$ is an oriented orthonormal basis. Note that $G_{2}$ also preserves the $4$-form
$$\ast_{\phi}\phi=e^{4567}+e^{2367}+e^{2345}+e^{1357}-e^{1346}-e^{1256}-e^{1247},$$
where $\ast_{\phi}$ is the Hodge star operator determined by the metric and orientation.

For a smooth $7$-manifold $M$ and a point $x\in M$, we define as in \cite{local curvature of G2,LW17}
\begin{align}
    \wedge_{+}^{3}(T_{x}^{\ast}M):=\left\{\varphi_{x}\in\wedge^{3}(T_{x}^{\ast}M)\Big{|} u^{\ast}\phi=\varphi_{x},
    \text{for\ invertible}\ u\in\text{Hom}_{\mathbb{R}}(T_{x}M,\mathbb{R}^{7})\right\}\notag
\end{align}
and the bundle
\begin{align}
    \wedge_{+}^{3}(T^{\ast}M):=\bigcup_{x\in M}\wedge_{+}^{3}(T_{x}^{\ast}M)\notag.
\end{align}
We call a section $\varphi$ of $\wedge_{+}^{3}(T^{\ast}M)$ a {\it positive $3$-form} on $M$ or a {\it $G_{2}$-structure} on $M$, and denote the space of positive $3$-forms by $\Omega^{3}_{+}(M)$. The existence of $G_{2}$-structure is equivalent to the
property that $M$ is both orientable and spinnable.

For a {positive} $3$-form $\varphi$, we define a $\Omega^{7}(M)$-valued bilinear form $\text{B}_{\varphi}$ by
$$\text{B}_{\varphi}(u,v)=\frac{1}{6}(u\lrcorner\varphi)\wedge(v\lrcorner\varphi)\wedge\varphi,$$
where $u, v$ are tangent vectors on $M$ and $``\lrcorner"$ is the interior multiplication operator (\textit{Here we use the orientation in \cite{Bry06}}). Then we can see that any $\varphi\in\Omega^{3}_{+}(M)$ determines a Riemannian metric $g_{\varphi}$ and an orientation $d V_{\varphi}$, hence the Hodge star operator $\ast_{\varphi}$ and the associated $4$-form
$$\psi:=\ast_{\varphi}\varphi$$
can also be uniquely determined by $\varphi$. 



\subsection{$G_{2}$-decomposition and torsion tensor}\label{subsection2.2} 

The group $G_{2}$ acts irreducibly on $\mathbb{R}^{7}$ (and hence on $\wedge^{1}(\mathbb{R}^{7})^{\ast}$ and $\wedge^{6}(\mathbb{R}^{7})^{\ast}$), but it
acts reducibly on $\wedge^{k}(\mathbb{R}^{7})^{\ast}$ for $2\leq k\leq 5$. Hence a $G_{2}$ structure $\varphi$ induces splittings
of the bundles $\wedge^{k}(T^{\ast}M)(2\leq k\leq5)$ into direct summands, which we denote by
$\wedge^{k}_{l}(T^{\ast}M,\varphi)$ with $l$ being the rank of the bundle. We let the space of {smooth} sections
of $\wedge^{k}_{l}(T^{\ast}M,\varphi)$ be $\Omega^{k}_{l}(M)$. Define the natural projections
$$\pi^{k}_{l}:\Omega^{k}(M)\longrightarrow \Omega^{k}_{l}(M),\ \ \alpha\longmapsto \pi^{k}_{l}(\alpha).$$
Then we have
\begin{align}
    \Omega^{2}(M)&=\Omega^{2}_{7}(M)\oplus\Omega^{2}_{14}(M),\notag\\
    \Omega^{3}(M)&=\Omega^{3}_{1}(M)\oplus\Omega^{3}_{7}(M)\oplus\Omega^{3}_{27}(M)\notag.
\end{align}
where each component is determined by
\begin{align}
    \Omega^{2}_{7}(M)&=\{X\lrcorner\varphi:X\in {\Gamma}(TM)\}=\{\beta\in\Omega^{2}(M):\ast_{\varphi}(\varphi\wedge\beta)=2\beta\},\notag\\
    \Omega^{2}_{14}(M)&=\{\beta\in\Omega^{2}(M):\psi\wedge\beta=0\}=\{\beta\in\Omega^{2}(M):\ast_{\varphi}(\varphi\wedge\beta)=-\beta\},\notag
\end{align}
and
\begin{align}
    \Omega^{3}_{1}(M)&=\{f\varphi:f\in C^{\infty}(M)\},\notag\\
    \Omega^{3}_{7}(M)&=\{\ast_{\varphi}(\varphi\wedge\alpha):\alpha\in\Omega^{1}(M)\}=\{X\lrcorner\psi:X\in {\Gamma}(TM)\},\notag\\
    \Omega^{3}_{27}(M)&=\{\eta\in\Omega^{3}(M):\eta\wedge\varphi=\eta\wedge\psi=0\}.\notag
\end{align}

\begin{remark}\label{remark2.2}
    $\Omega^{4}$ and $\Omega^{5}$ have the corresponding decompositions by Hodge duality. The more details for $G_{2}$-decomposition see \cite{Bry06,Spiros introduce to G2}.
\end{remark}

By the definition  of $G_{2}$ decomposition, we can find unique differential forms
$\tau_{0}\in {C^{\infty}}(M),\tau_{1},\widetilde{\tau}_{1}\in\Omega^{1}(M),\tau_{2}\in\Omega^{2}_{14}(M)$ and $\tau_{3}\in\Omega^{3}_{27}(M)$ such that (see \cite{Bry06})
\begin{align}
    d\varphi&=\tau_{0}\psi+3\!\ \tau_{1}\wedge\varphi+\ast_{\varphi}\tau_{3},\notag\\
    d\psi&=4\!\ \widetilde{\tau}_{1}\wedge\psi+\tau_{2}\wedge\varphi.\notag
\end{align}
In fact, {Bryant \cite{Bry06}} proved that $\tau_{1}=\widetilde{\tau}_{1}$. We call $\tau_{0}$ the \textit{scalar torsion}, $\tau_{1}$ the \textit{vector torsion}, $\tau_{2}$ the \textit{Lie algebra torsion}, and $\tau_{3}$ the \textit{symmetric traceless torsion}. We also call $\tau_{\varphi}:=\{\tau_{0},\tau_{1},\tau_{2},\tau_{3}\}$ the intrinsic torsion forms of the $G_{2}$-structure $\varphi$.
From \cite{Spiros introduce to G2}, the full torsion tensor $\mathbf{T}=\mathbf{T}_{ij}dx^{i}\otimes dx^{j}$ satisfies the followings:
\begin{align}
    \nabla_{i}\varphi_{jkl}&=\mathbf{T}_{i}^{\ m}\psi_{mjkl},\notag
\end{align}
The full torsion tensor $\mathbf{T}_{ij}$ is related to the
intrinsic torsion forms by the following:
\begin{align}\label{Def of T}
\mathbf{T}_{ij}=\frac{\tau_{0}}{4}g_{ij}-(\tau_{27})_{ij}-(\tau_{1}^{\#}\lrcorner\varphi)_{ij}-\frac{1}{2}(\tau_{2})_{ij}
\end{align}
where $(\tau_{1}^{\#}\lrcorner\varphi)_{ij}=(\tau_{1}^{\#})^{l}\varphi_{lij}$, $\tau_{27}\in\Gamma(S_{0}^{2}T^{\ast}M)$ satisfies $\tau_{3}=\tau_{27}\diamond\psi$ and and $\#$ is the isomorphism from $1$-form to vector fields.

${}$

If $\varphi$ is closed, which means $d\varphi=0$, then $\tau_{0},\tau_{1},\tau_{3}$ are all zero from \eqref{Def of T}, so the only nonzero torsion form is 
$$\tau\equiv\tau_{2}=\frac{1}{2}(\tau_{2})_{ij}dx^{i}{\wedge} dx^{j}=\frac{1}{2}\tau_{ij}dx^{i}{\wedge} dx^{j}.$$
Then we have 
$$
\mathbf{T}_{ij}=-\mathbf{T}_{ji}=-\frac{1}{2}(\tau_{2})_{ij}\ \ \text{or equivalently}\ \ \mathbf{T}=-\frac{1}{2}\tau,
$$
so that $\mathbf{T}$ is a skew-symmetric 2-tensor, i.e., a $2$-form.

\subsection{Curvature tensor}\label{subsection2.4}
Since $\varphi$ determines the unique metric $g\equiv g_{\varphi}$ on $M$, we obtain the Levi-Civita connection $\nabla\equiv\nabla_{g}$ induced by $g$, and also the Hodge Laplacian of $g$
$$\Delta_{\varphi}=dd^{\ast}_{\varphi}+d^{\ast}_{\varphi}d. $$

For a closed $G_{2}$-structure $\varphi$, the Ricci curvature is given by (see \cite{LW17})
\begin{align}
    R_{ik}=\nabla_{j}\mathbf{T}_{li}\varphi_{k}^{\ jl}-\mathbf{T}_{i}^{\ j}\mathbf{T}_{jk}.
\end{align}
The scalar curvature $R$ satisfies
\begin{align}\label{eq:scalar-torsion}
    R=g^{ij}R_{ij}=-|\mathbf{T}|^{2}=-g^{ik}g^{jl}\mathbf{T}_{ij}\mathbf{T}_{kl},
\end{align}
where $|\cdot|$ is the norm induced by metric $g$. Moreover, the factor $\nabla_{i}\mathbf{T}_{jm}$ can be expressed as

\begin{lemma}\label{lemma2.4}
    For any closed $G_{2}$-structure $\varphi$, we have
\begin{align}
    \nabla_{i}\mathbf{T}_{jk}&=\frac{1}{4}R_{ijmn}\varphi_{k}^{\ mn}+\frac{1}{4}R_{kjmn}\varphi_{i}^{\ mn}-\frac{1}{4}R_{ikmn}\varphi_{j}^{\ mn}\notag\\
    &\quad -\frac{1}{2}\mathbf{T}_{im}\mathbf{T}_{jn}\varphi_{k}^{\ mn}-\frac{1}{2}\mathbf{T}_{km}\mathbf{T}_{jn}\varphi_{i}^{\ mn}+\frac{1}{2}\mathbf{T}_{im}\mathbf{T}_{kn}\varphi_{j}^{\ mn}\notag.
\end{align}
\end{lemma}

\begin{proof}
    This lemma has been proved in \cite{LW17}.
\end{proof}

Let $\cK$ denote the vector space of algebraic curvature tensors in
seven dimensions.  The ordinary orthogonal decomposition is
\[
  \Rm=W+\frac15E\circ g+\frac{R}{84}g\circ g,
  \qquad
  E=\Ric-\frac17Rg,
\]
where $\circ$ is the Kulkarni--Nomizu product.  Cleyton-Ivanov \cite{CI08} proved that the $G_2$-structure
splits the Weyl tensor further.

\begin{lemma}[\cite{CI08}]
For a $G_2$-structure $\varphi$, the bundle of algebraic curvature
tensors decomposes orthogonally into irreducible $G_2$-modules as
\begin{align}
  \cK=\mathbb R\oplus S^2_0\oplus\mathcal W_{27}\oplus\mathcal W_{64}
   \oplus\mathcal W_{77}.\notag
\end{align}
Equivalently, the Riemann curvature tensor has the unique decomposition
\begin{equation}\label{eq:Riemann-G2-decomposition}
  \Rm
  =\frac{R}{84}g\circ g
   +\frac15E\circ g
   +W_{27}+W_{64}+W_{77}.
\end{equation}
where $W_{27}, W_{64}$ and $W_{77}$ are all elements of the projections of the Weyl tensor onto the respective subspaces.
\end{lemma}
Cleyton-Ivanov \cite{CI08} showed that $W_{27}$ satisfies
\begin{equation}
\label{eq:W27-norm}
|W_{27}|^2 = \frac{15}{28} |{{\rm Ric}}^{W}|^2.
\end{equation}
where the tensor ${{\rm Ric}}^{W}$ is $W$-Ricci tensor introduced in \cite{CI08}, in the closed case is controlled by Ricci curvature and the torsion tensor. The $64$-dimensional Weyl component satisfies
\begin{equation}
\label{eq:W64-bound}
|W_{64}| \leq C(|\nabla\mathbf{T}|+|\mathbf{T}|^{2}).
\end{equation}

From \cite{LW17}, we see that the associated metric tensor $g(t)$ evolves by
\begin{align}\label{eq:metric-flow-Chen}
\partial_{t}g_{ij}=2h_{ij},
\end{align}
where
$$
h_{ij}=-R_{ij}-\frac{1}{3}|\mathbf{T}|^{2}g_{ij}-2\widehat{\mathbf{T}}_{ij}
$$
and $\widehat{\mathbf{T}}_{ij}=\mathbf{T}_{i}^{\ k}\mathbf{T}_{kj}$.
Moreover,
\begin{equation}\label{eq:volume-evolution}
  \partial_t\dd V_g
  =\frac23|T|^2\dd V_g
  =-\frac23R\,\dd V_g.
\end{equation}

\section{Proof of the spatial $L^p$ extension criterion}\label{section3}

In this section, wIn this section we prove the spatial \(L^p\) extension criterion, which can be divided into two parts: $p>\frac{7}{2}$ in Section \ref{section3.1} and $p=\frac{7}{2}$ in Section \ref{section3.2}.
\subsection{$p>\frac{7}{2}$}\label{section3.1}
As \cite{LW17}, we give the definition of $\Lambda$
$$\Lambda(x,t):=(|{\rm Rm}(x,t)|^{2}_{g(t)}+|\nabla\mathbf{T}(x,t)|^{2}_{g(t)})^{\frac{1}{2}}.$$
Next, we recall some lemmas to prove the case $p>\frac{7}{2}$.
\begin{lemma}[Shi-type estimate,\cite{LW17}]\label{Shi}
    Suppose that $K>0$ and $\varphi(t)$ is a solution of the Laplacian flow for closed $G_{2}$ structures on a compact manifold $M$ for $t\in[0,\frac{1}{K}]$. For all $k\in\mathbb{N}$, there exists a constant $C_{k}$ such that if $\Lambda(x,t)\leq K$ on $M\times [0,\frac{1}{K}]$, then 
    $$|\nabla^{k}{\rm Rm}(x,t)|_{g(t)}+|\nabla^{k+1}\mathbf{T}(x,t)|_{g(t)}\leq C_{k}t^{-\frac{k}{2}}K,\quad\quad t\in\left(0,\frac{1}{K}\right].$$
\end{lemma}

\begin{lemma}[\cite{LLM26}]\label{3.2}
Let $(N, g)$ be complete and let $\alpha$ be a smooth tensor field satisfying
\[
|\alpha|_{g} \le c_{0}, \qquad |\nabla_{g}^2 \alpha|_{g} \le c_{2}.
\]
Then
$$
|\nabla_{g} \alpha|^2_{g} \le 2nc_{0} c_{2}.
$$
where $c_{0},c_{2}$ are nonnegative constants and $n$ is the dimension of manifold. 
\end{lemma}

\begin{lemma}\label{noncollapse}
Let $(M^7,\varphi(t))_{t\in[0,T)}$ be a solution of
\eqref{eq:flow} on a compact manifold, where $0<T<\infty$. If the scalar curvature satisfies
\begin{align}\label{eq:scalar-assumption}
    \sup_{M\times[0,T)}(-R)\leq K<\infty,
\end{align}
then there exist numbers
$t_\ast\in(0,T)$, $\rho_0>0$, and $\kappa>0$, depending only on the
initial flow, $K$, and $T$, such that, for every $t\in[t_\ast,T)$,
\begin{equation}
  \Vol_{g(t)}B_{g(t)}(x,r)\geq\kappa r^7
\end{equation}
for all $x\in M$ and all $0<r\leq\rho_0$.
\end{lemma}

\begin{proof}
By \eqref{eq:scalar-torsion} and \eqref{eq:scalar-assumption},
$|T|^2\leq K$, so the error tensor 
$$E=-\frac{2}{3}|\mathbf{T}|^{2}g_{ij}-4\widehat{\mathbf{T}}_{ij}$$ in
\eqref{eq:metric-flow-Chen} is uniformly bounded.  Chen's theorem gives
$\kappa$-non-collapsing relative to an upper scalar-curvature bound on
a fixed scale.  Since $R\leq0$, the required condition
$R\leq r^{-2}$ on a ball of radius $r$ is automatic.
\end{proof}
Now, we can prove the main theorem when $p>\frac{7}{2}$.
\begin{theorem}\label{3.4}
Let $(M^7,\varphi(t))_{t\in[0,T)}$ be a solution of
\eqref{eq:flow} on a compact manifold, where $0<T<\infty$. For some $p>7/2$,  assume
that the scalar curvature and the $L^{p}$ norm of $W_{77}$ satisfy
$$
  \sup_{M\times[0,T)}(-R)\leq K<\infty,
$$
and
$$
  \sup_{0\leq t<T}\int_M |W_{77}(t)|^p\,\dd V_{g(t)}\leq A<\infty.
$$
Then $\varphi(t)$ extends smoothly past time $T$.
\end{theorem}
\begin{proof}
We prove this theorem by contradiction.
Suppose that \(T<\infty\) is the maximal existence time. Since bounded Riemann curvature and $\nabla\mathbf{T}$ give extension, which means
$$\limsup_{t\rightarrow T}\max_{M}\left(|{\rm Rm}|+|\nabla\mathbf{T}|\right)=+\infty.$$
There exists a sequence of times $t_i \rightarrow T$ and points $x_i \in M$ such that
\begin{align}
Q_i := \Lambda(x_i, t_i) = \max_{M \times [0, t_i]} \left(|{\rm Rm}|^{2}+|\nabla\mathbf{T}|^{2}\right)^{\frac{1}{2}} \longrightarrow \infty.\notag
\end{align}
For large $i$, define on $[-1, 0]\subset[-t_{i}Q_{i},0]$
\begin{equation}
\label{eq:rescaled-flow}
\varphi^{(i)}(s) = Q_{i}^{3/2} \varphi\!\left( t_i + \frac{s}{Q_{i}} \right), \qquad
g^{(i)}(s) = Q_{i}\, g\!\left( t_i + \frac{s}{Q_{i}} \right).
\end{equation}
Then
\begin{align}\label{eq:rescaled-Rm}
|\Lambda^{(i)}| \le 1 \quad \text{on } M \times [-1, 0], \qquad
|\Lambda^{(i)}|(x_i, 0) = 1.
\end{align}
Since the Laplacian flow is invariant under the scaling in \eqref{eq:rescaled-flow}, Lemma \ref{Shi} yields
\begin{equation}
\label{eq:rescaled-nabla2T}
\sup_M \left(|(\nabla^{(i)})^{m}{\rm Rm}^{(i)}|(0)+|(\nabla^{(i)})^{m+1} \mathbf{T}^{(i)}|(0) \right)\le C_{m},\quad m\geq 0
\end{equation}
with a positive constant $C_{m}$. 

From the assumption of scalar curvature, for $x\in M$, Lemma \ref{noncollapse} gives
\begin{align}\label{volume}
   \Vol_{g^{(i)}}B_{g^{(i)}}(x,r)\geq\kappa r^7, \qquad 0<r\leq\rho_0\sqrt{Q_i}. 
\end{align}
From \eqref{eq:rescaled-Rm}, we know that the Riemann curvature satisfies
$$|{\rm Rm}^{(i)}(x,0)|\leq \sup_{M\times[-1,0]}|\Lambda^{(i)}(x,t)|\leq 1.$$
Then by \eqref{volume} and Theorem 5.42 in \cite{CLN06}, we have a uniform injectivity radius estimate 
$${\rm inj}(M,g^{(i)}(0),x_{i})\geq c$$
with constant $c>0$. Applying Theorem 1.5 in \cite{LW17}, we obtain a subsequence of $(M,\varphi^{(i)}(t),x_{i}) $ converging to a limit $(M_{\infty},\varphi^{(\infty)}(t),x_{\infty})$ with $|\Lambda^{(\infty)}(x_{\infty},0)|=1$.

The bounded scalar curvature gives
\begin{equation}
\label{eq:rescaled-T}
\sup_{M \times [-1,0]} |\mathbf{T}^{(i)}|^2 \le \frac{K}{Q_i}.
\end{equation}
Applying Lemma \ref{3.2} at time zero gives
\begin{equation}
\label{eq:rescaled-nablaT-vanish}
\sup_{M}|\nabla^{(i)} \mathbf{T}^{(i)}(0)| \le C\, Q_i^{-1/4} \longrightarrow 0,
\end{equation}
where $C$ is a uniformly constant. It follows that
\begin{equation}
\label{eq:rescaled-ric-vanish}
\sup_{M}|{{\rm Ric}}^{(i)}(0)| \longrightarrow 0,
\end{equation}
 Then Equations \eqref{eq:W27-norm} and \eqref{eq:W64-bound} imply
\begin{equation}
\label{eq:rescaled-W27-W64-vanish}
\sup_{M}\left(|W^{(i)}_{27}(0)| + |W^{(i)}_{64}(0)|\right) \longrightarrow 0.
\end{equation}
From the assumption of $L^{p}$ norm of $W_{77}$, for $p>\frac{7}{2}$, we have
\begin{align}
    \int_M|W_{77}^{(i)}|^p\,\dd V_{g_{i}}
  =Q_{i}^{\frac{7}{2}-p}
  \int_M\left|W_{77}\left(t_{i}+\frac{s}{Q_{i}}\right)\right|^p\,\dd V_{g(t_{i}+\frac{s}{Q_{i}})}\leq AQ_{i}^{\frac{7}{2}-p}\longrightarrow 0\notag
\end{align}
as $i\rightarrow\infty$. Let $\Omega \subset M_{\infty}$ be a relatively compact neighbourhood
of $x_\infty$, and let $F_i:\Omega\to M$ be the convergence
embeddings. Since the projections defining $W_{77}$ depend
smoothly on the metric and the positive three-form,
$$F_i^{\ast}W_{77}^{(i)}\longrightarrow W_{77}^{(\infty)}$$
smoothly on $\Omega$. Hence
$$\int_\Omega |W_{77}^{(\infty)}|^p\,dV_{g^{(\infty)}}=\lim_{i\to\infty}\int_\Omega |F_i^{\ast}W_{77}^{(i)}|^p\,dV_{F_i^{\ast}g^{(i)}}=0.$$
Thus $W_{77}^{(\infty)}=0$ on $\Omega$.

 Letting $i\rightarrow\infty$ and from the decomposition of Riemann curvature \eqref{eq:Riemann-G2-decomposition}, we get 
 $$|{\rm Rm}^{(\infty)}|(x_{\infty},0)=0,$$ 
which contradicts 
$$|{\rm Rm}^{(\infty)}|(x_{\infty},0)=|\Lambda^{(\infty)}|(x_{\infty},0)=1.$$
Then we prove the theorem.
\end{proof}

\subsection{Critical case: $p=\frac{7}{2}$}\label{section3.2}

We record the torsion-free specialization of the compactness theorem of
Theorem 1.10 in \cite{LMZ26}.
\begin{lemma}\label{thm:LMZ-special}
Let $(N_i^7,\varphi_i,h_i,p_i)$ be complete torsion-free $G_2$-manifolds.
Assume that there is $\kappa>0$ such that
\begin{align}
  \Vol_{h_i}B_{h_i}(x,r)\geq\kappa r^7
  \qquad
  \text{for all }x\in N_i\text{ and }0<r\leq1,\notag
\end{align}
and that, for every $r>0$,
\begin{equation}\label{eq:LMZ-energy}
  \int_{B_{h_i}(p_i,r)}|\Rm(h_i)|^{7/2}\,\dd V_{h_i}
  \leq E(r)<+\infty.
\end{equation}
Then a subsequence converges in the pointed $C^\infty$ topology to a smooth torsion-free $G_2$-manifold.
\end{lemma}

\begin{proof}
From the assumption, $\varphi_{i}$ is a torsion-free $G_{2}$ structure, thus
$$\Delta_{\varphi_{i}}\varphi_{i}=d^{\ast}_{h_{i}}d\varphi_{i}+dd^{\ast}_{h_{i}}\varphi_{i}=0,$$
$\varphi_{i}$ can be viewed as a complete steady gradient $G_2$-soliton with potential $f_i\equiv0$ and soliton constant $\lambda_i=0$.  

Since a torsion-free $G_{2}$ structure means $h_i$ is Ricci-flat, and the unit balls are uniformly noncollapsed, the local Sobolev constants are uniformly bounded. The corresponding local logarithmic Sobolev inequality yields a uniform lower bound for the local entropy. Thus the sequence lies in a fixed class
$\cM(\Lambda_0,F_0,\underline{\nu_0})$ of Definition 1.3 in \cite{LMZ26}. Note that $f_i\equiv0$, \eqref{eq:LMZ-energy} and Theorem 1.10 in \cite{LMZ26} show that the subsequence converges in the pointed $C^\infty$ topology to a smooth torsion-free $G_2$-manifold.
\end{proof}
Then we prove the main theorem for $p=\frac{7}{2}$.
\begin{theorem}\label{3.6}
Let $(M^7,\varphi(t))_{t\in[0,T)}$ be a solution of
\eqref{eq:flow} on a compact manifold, where $0<T<\infty$. Suppose
that the scalar curvature and the $L^{\frac{7}{2}}$ norm of $W_{77}$ satisfy
$$
  \sup_{M\times[0,T)}(-R)\leq K<\infty,
$$
and
$$
  \sup_{0\leq t<T}\int_M |W_{77}(t)|^\frac{7}{2}\,\dd V_{g(t)}\leq A<\infty.
$$
Then $\varphi(t)$ extends smoothly past time $T$.
\end{theorem}
\begin{proof}
We also prove this theorem by contradiction. Repeating the blow-up construction in the proof of Theorem \ref{3.4}, suppose that \(T<\infty\) is the maximal existence time, then there exists a sequence of times $t_i \rightarrow T$ and points $x_i \in M$ such that
\begin{equation}
Q_i := \Lambda(x_i, t_i) = \max_{M \times [0, t_i]} \left(|{\rm Rm}|^{2}+|\nabla\mathbf{T}|^{2}\right)^{\frac{1}{2}} \longrightarrow \infty.
\end{equation}
For large $i$, define on $[-1, 0]\subset[-t_{i}Q_{i},0]$
\begin{equation}
\varphi^{(i)}(s) = Q_{i}^{3/2} \varphi\!\left( t_i + \frac{s}{Q_{i}} \right), \qquad
g^{(i)}(s) = Q_{i}\, g\!\left( t_i + \frac{s}{Q_{i}} \right).
\end{equation}
By the blow-up construction above, we obtain a subsequence of $(M,\varphi^{(i)}(t),x_{i}) $ converging to a limit $(M_{\infty},\varphi^{(\infty)},x_{\infty})$ with 
$$|\Lambda^{(\infty)}(x_{\infty},0)|=1,\quad |{\rm Rm}^{(\infty)}|(x,0)=|W_{77}^{(\infty)}|(x,0),\quad |\mathbf{T}^{(\infty)}(x,0)|=0.$$
For any relatively compact domain $\Omega$ in the limit, critical scale invariance gives
\begin{align}
    \int_\Omega|W_{77}^{(i)}|^\frac{7}{2}\,\dd V_{g^{(i)}}
  &\leq \int_M |W_{77}^{(i)}|^\frac{7}{2}\,\dd V_{g^{(i)}}\\
  &=Q_{i}^{\frac{7}{2}-\frac{7}{2}}
  \int_M\left|W_{77}\left(t_{i}+\frac{s}{Q_{i}}\right)\right|^{\frac{7}{2}}\,\dd V_{g^{(t_{i}+s/Q_{i})}}\notag\\
  &\leq A<\infty,\notag
\end{align}
First pass to the limit on each compact region and then exhaust $M_{\infty}$. This yield
$$\int_{M_{\infty}}|{\rm Rm}^{(\infty)}|^\frac{7}{2}\,\dd V_{g^{(\infty)}}=\int_{M_{\infty}}|W_{77}^{(\infty)}|^\frac{7}{2}\,\dd V_{g^{(\infty)}}\leq A<\infty.$$

Next, we show $(M_{\infty},g^{(\infty)})$ is flat. We also prove it by contradiction. Fix $x\in M_{\infty}$.  Suppose that $|\Rm^{(\infty)}|(x)>0$.  Choose $a_j\to\infty$
and form the blow-down sequence
$$
  g_j:=a_j^{-2}g^{(\infty)},
  \qquad
  \varphi_j:=a_j^{-3}\varphi^{(\infty)}.
$$
Since $\varphi^{(\infty)}$ is a torsion-free $G_{2}$-structure, $(M_{\infty},\varphi_j,g_j)$ can be viewed as a complete torsion-free steady gradient $G_2$-soliton with zero potential.  Non-collapsing \eqref{volume}
$${\rm Vol}_{g^{(\infty)}}B_{g^{(\infty)}}(y,r)\geq \kappa r^{7}\quad \text{for all}\  y\in M_{\infty},\ r>0$$
and critical energy are scale invariant, and for every fixed $r>0$,
\begin{align}\label{eq:blowdown-energy}
  \int_{B_{g_j}(x,r)}|\Rm(g_j)|^{\frac{7}{2}}\,\dd V_{g_j}
  &=\int_{B_{g^{\infty}}(x,a_jr)}|\Rm^{(\infty)}|^{\frac{7}{2}}\,\dd V_{g^{\infty}}\\
  &\leq\int_{M_{\infty}}|\Rm(g^{\infty})|^{\frac{7}{2}}\,\dd V_{g^{(\infty)}}.\notag
\end{align}
Lemma \ref{thm:LMZ-special} gives a pointed smoothly convergent
subsequence, so the basepoint curvature must remain bounded.  But
\begin{equation}\label{eq:basepoint-blowup}
  |\Rm(g_j)|_{g_j}(x)
  =a_j^2|\Rm^{(\infty)}|_{g^{(\infty)}}(x)\longrightarrow\infty.
\end{equation}
This contradicts the pointed smooth convergence, which yields a uniform curvature bound at the basepoint.  Hence $|\Rm^{(\infty)}|(x)=0$, and $x$ was arbitrary, which implies $(M_{\infty},g^{(\infty)})$ is flat.

By the blow-up construction above, we also know 
$$|{\rm Rm}^{(\infty)}|(x_{\infty},0)=|\Lambda^{(\infty)}|(x_{\infty},0)=1.$$
Thus, we get a contradiction. We prove Theorem \ref{3.6}.
\end{proof}

\section{Mixed spacetime bounds of $W_{77}$} \label{section4}
In this section, we establish extension criteria in terms of mixed
space-time norms of $W_{77}$. This criterion allows the spatial \(L^p\) norm of \(W_{77}\) to grow in time, provided that its mixed space-time norm remains finite.
\begin{lemma}\label{lem4.1}
Let $(M^7,\varphi(t))_{t\in[0,T)}$ be a solution of \eqref{eq:flow} on a compact manifold, where $0<T<\infty$. Suppose that the scalar curvature  satisfies
$$\sup_{M\times[0,T)}(-R)\leq K<\infty,$$
for any $t_{0}\in (0,T)$, if there exists $Q$ such that
$Q\geq K$ and 
$$|{\rm Rm}(x,t)|\leq Q\quad \text{on}\ M\times[t_{0}-2Q^{-1},t_{0}],$$
then for any $m\geq 0$, we have
\begin{align}\label{4.1}
    \sup_{M}|\nabla^{m}\mathbf{T}|(t_{0})\leq C_{m+1}\sqrt{K}Q^{\frac{m}{2}},\quad\quad \sup_{M}|\partial_{t}|{\rm Rm}||(t_{0})\leq C\sqrt{K}Q^{\frac{3}{2}}.
\end{align}
where $C_{k+2}$ and $C$ are some uniform constants independent of $K$ and $Q$.
\end{lemma}
\begin{proof}
    From the assumption of Riemann curvature, for $s\in[-2,0]$, we define $\tilde{\varphi}$ and $\tilde{g}$ as follows
    \begin{align}\label{4.2}
        \tilde{\varphi}(s)=Q^{\frac{3}{2}}\varphi(t_{0}+sQ^{-1}),\quad \tilde{g}(s)=Qg(t_{0}+sQ^{-1}).
    \end{align}
    Thus, we have
    $$|\widetilde{{\rm Rm}}|(x,s)\leq 1,\quad |\tilde{\mathbf{T}}|^{2}(x,s)\leq \frac{K}{Q},\quad \text{on}\ M\times[-2,0].$$
    Since the Laplacian flow is invariant under the scaling in \eqref{4.2}, by Lemma \ref{lemma2.4} and Shi-type estimate in Lemma \ref{Shi}, for any $k\in\mathbb{N}^{+}$, we have
    $$|\tilde{\nabla}^{k}\widetilde{\rm Rm}|(x,s)+|\tilde{\nabla}^{k+1}\tilde{\mathbf{T}}|(x,s)\leq C_{m},\quad \text{on}\ M\times[-1,0].$$
    From (4.23) in \cite{LW17}, we have
    \begin{align}
        (\partial_{s}-\tilde{\triangle})|\tilde{\nabla}^{k+1}\tilde{\mathbf{T}}|^{2}&=-2|\tilde{\nabla}^{k+2}\tilde{\mathbf{T}}|^{2}+\sum_{i=0}^{k+1}\tilde{\nabla}^{k+1}\tilde{\mathbf{T}}\tilde{\ast}\tilde{\nabla}^{k+1-i}\tilde{\mathbf{T}}\tilde{\ast}\tilde{\nabla}^{i}(\widetilde{\rm Rm}+\tilde{\mathbf{T}}\tilde{\ast}\tilde{\mathbf{T}})\notag\\
        &\quad+\sum_{i=0}^{k+1}\tilde{\nabla}^{k+1}\tilde{\mathbf{T}}\tilde{\ast}\tilde{\nabla}^{k+1-i}(\widetilde{\rm Rm}\tilde{\ast}\tilde{\mathbf{T}})\tilde{\ast} \tilde{\nabla}^{i}\tilde{\psi}\notag\\
        &\quad+\sum_{i=0}^{k+1}\tilde{\nabla}^{k+1}\tilde{\mathbf{T}}\tilde{\ast}\tilde{\nabla}^{k+1-i}(\tilde{\nabla}\tilde{\mathbf{T}}\tilde{\ast}\tilde{\mathbf{T}})\tilde{\ast} \tilde{\nabla}^{i}\tilde{\varphi}\notag\\
        &\leq -|\tilde{\nabla}^{k+2}\tilde{\mathbf{T}}|^{2}+C_{k+2}\sum_{i=0}^{k+1}|\tilde{\nabla}^{i}\tilde{\mathbf{T}}|^{2}\notag.
    \end{align}
    Letting $\theta=s+1\in[0,1]$, and define
$$F_{k+1}(s)=\sum_{i=0}^{k+1}b_{i}\theta^{i}(s)|\tilde{\nabla}^{i}\tilde{\mathbf{T}}|^{2}(s)$$
Set $b_{k+1}=1$ and choose the coefficients backwards so that $b_{i-1}\geq 2ib_{i}$. The term $ib_{i}\theta^{i-1}|\tilde{\nabla}^{i}\tilde{\mathbf{T}}|^{2}$ arising from differentiating the time weight is then absorbed by the negative term at the preceding level. Since $0\leq\theta\leq 1$, the remaining terms are bounded by $C_{k+2}F_{k+1}$, and hence
$$(\partial_{s}-\tilde{\triangle})F_{k+1}\leq C_{k+2}F_{k+1}.$$
When $\theta=0$, one has $F_{k+1}(-1)=b_{0}|\tilde{\mathbf{T}}|^{2}\leq b_{0}\frac{K}{Q}$. Then by the maximum principle, we get $F_{k+1}\leq C\frac{K}{Q}$, which means
$$|\tilde{\nabla}^{k+1}\tilde{\mathbf{T}}|\leq C_{k+2}\frac{\sqrt{K}}{\sqrt{Q}}\quad \text{at}\ s=0.$$
Thus, scaling yields
\begin{align}\label{4.3}
    \sup_{M}|\nabla^{k+1}\mathbf{T}|(t_{0})\leq C_{k+2}\sqrt{K}Q^{\frac{k+1}{2}}.
\end{align}

Since ${\rm Ric}=\nabla\mathbf{T}\ast\varphi+\mathbf{T}\ast\mathbf{T}$, then $\partial_{t}g$ satisfies
$$\partial_{t}g=\nabla\mathbf{T}\ast\varphi+\mathbf{T}\ast\mathbf{T}.$$
From \eqref{4.3}, it follows that
$$|\nabla^{2}(\partial_{t}g)|\leq C\sqrt{K}Q^{\frac{3}{2}}.$$
The general formula (2.66) in \cite{CLN06} for Riemann curvature gives
\begin{align}
    |\partial_{t}|{\rm Rm}||\leq C(|\nabla^{2}(\partial_{t}g)|+|{\rm Rm}|\cdot |\partial_{t}g|)\leq C\sqrt{K}Q^{\frac{3}{2}}.\notag
\end{align}
Thus, we prove this Lemma.
\end{proof}

\begin{theorem}\label{thm4.2}
Let $(M^7,\varphi(t))_{t\in[0,T)}$ be a solution of
\eqref{eq:flow} on a compact manifold, where $0<T<\infty$. 
Let $\frac{7}{2}< p<\infty$ and $1\leq q<\infty$ satisfy
$$\frac{1}{q}+\frac{7}{p}\leq 2.$$
Suppose that the scalar curvature and $W_{77}$ satisfy
$$
  \sup_{M\times[0,T)}(-R)\leq K<\infty,
$$
and  
$$
  \int_{0}^{T}\left(\int_{M}|W_{77}|^{p}\dd V_{g(t)}\right)^{\frac{q}{p}}\dd t \leq A<\infty.
$$
Then $\varphi(t)$ extends smoothly past time $T$.
\end{theorem}

\begin{proof}
    We also prove this theorem by contradiction. Suppose that \(T<\infty\) is the maximal existence time, then we have
$$\limsup_{t\rightarrow T}\max_{M}|{\rm Rm}|=+\infty.$$ 
We fix a sufficiently large constant $Q_{0}$, and define
$$m(t):=\max_{M}|{\rm Rm}|(t),\quad\quad Q_{j}:=2^{j}Q_{0}.$$
Let $t_{j}$ be the first time at which $m(t)$ reaches $Q_{j}$, and choose $x_{j}\in M$ so that 
$$|{\rm Rm}|(x_{j},t_{j})=Q_{j}.$$
Thus, $t_{i}\rightarrow T$ as $i\rightarrow \infty$ and
$$|{\rm Rm}|(x,t)\leq Q_{j}\quad \text{on}\ M\times[0,t_{j}].$$

After discarding finitely many indices, we may assume that
$t_{j-1}\ge 2Q_j^{-1}$.
For every $s\in[t_{j-1},t_j]$, the interval
$[s-2Q_j^{-1},s]$ is contained in $[0,t_j]$.
Lemma \ref{lem4.1} applies at each such time $s$, with the same
constants. It follows that
\begin{align}
    \frac{Q_{j}}{2}=m(t_{j})-m(t_{j-1})\leq C\sqrt{K}Q_{j}^{\frac{3}{2}}(t_{j}-t_{j-1}),\notag
\end{align}
which implies
\begin{align}
    t_{j}-t_{j-1}\geq \frac{1}{2C\sqrt{K}\sqrt{Q_{j}}}\notag.
\end{align}
 Choose $\delta\leq \frac{1}{8C}$, and define $I_{j}$ as
 $$I_{j}:=\left[t_{j}-\frac{\delta}{\sqrt{K}\sqrt{Q_{j}}},t_{j}\right],$$
 we get $I_{j}\subset [t_{j-1},t_{j}]$ and the intervals are pairwise disjoint.

 Since $|\nabla{\rm Rm}|(t_{j})\leq CQ_{j}^{\frac{3}{2}}$, choose $\rho\leq \frac{1}{8C}$ and integrating along minimizing $g(t_{j})$-geodesics gives
 $$|{\rm Rm}|(x,t_{j})\geq \frac{7}{8}Q_{j}\quad \text on\  B_{g(t_{j})}\left(x_{j},\rho Q_{j}^{-\frac{1}{2}}\right).$$
 Integrating $\eqref{4.1}$ at the same fixed spatial point then yields
 $$|{\rm Rm}|(x,t)\geq \frac{3}{4}Q_{j}\quad \text on\  B_{g(t_{j})}\left(x_{j},\rho Q_{j}^{-\frac{1}{2}}\right)\times I_{j}.$$
 From Equations \eqref{eq:W27-norm} and \eqref{eq:W64-bound}, we get
 $$|W_{27}|+W_{64}|\leq C(|\nabla\mathbf{T}|+|\mathbf{T}|^{2})\leq C(\sqrt{K}Q_{j}^{\frac{1}{2}}+K)\quad \text{on}\  B_{g(t_{j})}\left(x_{j},\rho Q_{j}^{-\frac{1}{2}}\right)\times I_{j},$$
similar to Ricci curvature and scalar curvature, we get
 \begin{align}\label{equ}
      \frac{Q_{j}}{2}\leq|W_{77}|(x,t)\leq Q_{j}\quad \text{on}\  B_{g(t_{j})}\left(x_{j},\rho Q_{j}^{-\frac{1}{2}}\right)\times I_{j}.
 \end{align}

For convenience, we use $B_{j}$ instead of $B_{g(t_{j})}\left(x_{j},\rho Q_{j}^{-\frac{1}{2}}\right)$. From Lemma \ref{noncollapse}, we get
$${\rm Vol}_{g(t_{j})}B_{j}\geq \kappa\rho^{7}Q_{j}^{-\frac{7}{2}}.$$
Using $\partial_{t}\dd V_{g(t)}=\frac{2}{3}|\mathbf{T}|^{2}\dd V_{g(t)}$, for any $t\in I_{j}$, we have
$$\dd V_{g(t)}\geq e^{\frac{-2K(t_{j}-t)}{3}}\dd V_{g(t_{j})}.$$
Since $K|I_{j}|=\delta\frac{\sqrt{K}}{\sqrt{Q_{j}}}\leq\delta$, it follows that
\begin{align}\label{ball}
    {\rm Vol}_{g(t)}B_{j}\geq e^{-\frac{2\delta}{3}}\kappa\rho^{7}Q_{j}^{-\frac{7}{2}}.
\end{align}

Now, we use the above results to get a contradiction. For $p<\infty$ and $t\in I_{j}$, we have
\begin{align}
    \int_{M}|W_{77}|^{p}\dd V_{g(t)}\geq \int_{B_{j}}|W_{77}|^{p}\dd V_{g(t)}\geq \left(\frac{Q_{j}}{2}\right)^{p}e^{-\frac{2\delta}{3}}\kappa\rho^{7}Q_{j}^{-\frac{7}{2}}=2^{-p}e^{-\frac{2\delta}{3}}\kappa\rho^{7}Q_{j}^{p-\frac{7}{2}}\notag.
\end{align}
Consequently,
\begin{align}
    \int_{I_{j}}\left(\int_{M}|W_{77}|^{p}\dd V_{g(t)}\right)^{\frac{q}{p}}dt\geq c_{p,q}Q_{j}^{q-\frac{7q}{2p}}|I_{j}|=\frac{c_{p,q}}{\sqrt{K}}Q_{j}^{q-\frac{7q}{2p}-\frac{1}{2}},\notag
\end{align}
where $c_{p,q}=\delta(2^{-p}e^{-\frac{2\delta}{3}}\kappa\rho^{7})^{\frac{q}{p}}$. The condition $\frac{1}{q}+\frac{7}{p}\leq 2$ implies
$$q-\frac{7q}{2p}-\frac{1}{2}\geq 0.$$
Thus each $I_{j}$ carries a positive amount of energy independent of $j$; under strict inequality, this amount even grows with $j$. Since the intervals $I_{j}$ are pairwise disjoint, we get
\begin{align}
    \int_{0}^{T}\left(\int_{M}|W_{77}|^{p}\dd V_{g(t)}\right)^{\frac{q}{p}}dt\geq \sum_{j\gg1} \int_{I_{j}}\left(\int_{M}|W_{77}|^{p}\dd V_{g(t)}\right)^{\frac{q}{p}}dt=\infty,
\end{align}
which contradicts the assumption.
\end{proof}

\begin{corollary}
    Let $(M^7,\varphi(t))_{t\in[0,T)}$ be a solution of
\eqref{eq:flow} on a compact manifold, where $0<T<\infty$. 
Let $f:[0,\infty)\rightarrow[1,\infty)$ be finite-valued and nondecreasing, with
\begin{align}\label{4.7}
    \int_{1}^{\infty}\frac{\dd s}{sf(s)}=\infty.
\end{align}
Suppose that the scalar curvature and $W_{77}$ satisfy
$$
  \sup_{M\times[0,T)}(-R)\leq K<\infty
$$
with $K>0$ and  
\begin{align}\label{4.8}
    \int_{0}^{T}\int_{M}\frac{|W_{77}|^{4}}{f(|W_{77}|/K)}\dd V_{g(t)}\dd t \leq A<\infty.
\end{align}
Then $\varphi(t)$ extends smoothly past time $T$.
\end{corollary}
\begin{proof}
    From \eqref{equ} and \eqref{ball}, take the same notions in Theorem \ref{thm4.2}, we have
    $$\int_{I_{j}}\int_{B_{j}}\frac{|W_{77}|^{4}}{f(|W_{77}|/K)}\dd V_{g(t)}\dd t\geq \frac{\delta e^{-\frac{2}{3}\delta}\kappa\rho^{7}}{2^{4}}\cdot\frac{1}{\sqrt{K}f(Q_{j}/K)}.$$
    Let $b=\frac{Q_{0}}{K}>0$. For all sufficiently large $j$, monotonicity gives
    $$\int_{2^{j}b}^{2^{j+1}b}\frac{ds}{sf(s)}\leq \frac{\log 2}{f(2^{j}b)}.$$
    Hence \eqref{4.7} implies
    $$\sum_{j=1}^{\infty}\frac{\log 2}{f(2^{j}b)}=\infty.$$
    which contradicts \eqref{4.8}.
\end{proof}

\textbf{Acknowledgments.}\ \ 
 The first author is supported by the China Postdoctoral Science Foundation (2026M793350). The second author is funded by the Shanghai Institute for Mathematics and Interdisciplinary Sciences (SIMIS) under grant number SIMIS-ID-2025-AD. The third author is partially funded by NSFC grants No.12571091. He acknowledges support from the ICTP through the Associates Programme (2020-2025).

\end{document}